\documentclass{amsart}

\usepackage{amsfonts}
\usepackage{latexsym}
\usepackage{amsmath}
\usepackage{amssymb}
\usepackage{multicol}

\newcommand{\R}{\mathbb{R}}

\newcommand{\C}{\mathbb{C}}
\newcommand{\ct}{C_\circ }
\newcommand{\N}{\mathbb{N}}

\newcommand{\T}{\mathbb{T}}
\renewcommand{\P}{\mathbb{P}}
\newcommand{\card}{\mathrm{card}}

\newtheorem{theorem}{Theorem}
\newtheorem{lemma}[theorem]{Lemma}
\newtheorem{corollary}[theorem]{Corollary}

\theoremstyle{definition}
\newtheorem{definition}[theorem]{Definition}
\newtheorem{example}[theorem]{Example}
\newtheorem{proposition}[theorem]{Proposition}

\newtheorem{remark}[theorem]{Remark}

\numberwithin{equation}{section}
\numberwithin{theorem}{section}

\begin{document}

\title[Second Main Theorem in the Tropical Projective Space]{Second Main Theorem in the Tropical Projective Space}

\author{Risto Korhonen}
\address{Department of Physics and Mathematics, University of Eastern Finland, P.O. Box 111,
FI-80101 Joensuu, Finland}
\email{risto.korhonen@uef.fi}
\author{Kazuya Tohge}
\address{College of Science and Engineering, Kanazawa University, Kakuma-machi, Kanazawa, 920-1192, Japan}
\email{tohge@t.kanazawa-u.ac.jp}

\thanks{Supported in part by the Academy of Finland Grant \#268009, and the Japan Society for the Promotion of Science Grant-in-Aid for Scientific Research (C) \#22540181 and \#25400131.}


\begin{abstract}
Tropical Nevanlinna theory, introduced by Halburd and Southall as a tool to analyze integrability of ultra-discrete equations, studies the growth and complexity of continuous piecewise linear real functions. The purpose of this paper is to extend tropical Nevanlinna theory to $n$-dimensional tropical projective spaces by introducing a natural characteristic function for tropical holomorphic curves, and by proving a tropical analogue of Cartan's second main theorem. It is also shown that in the $1$-dimensional case this result implies a known tropical second main theorem due to Laine and Tohge.
\end{abstract}

\maketitle

\section{Introduction}

Tropical Nevanlinna theory of piecewise linear real functions, or of \textit{tropical meromorphic functions}, was recently introduced by Halburd and Southall \cite{halburds:09}. They defined tropical versions of the Nevanlinna functions, and showed that they share many of the properties of their classical counterparts \cite{hayman:64,cherryy:01}, including Jensen's formula and an analogue of the first main theorem. Halburd and Southall applied tropical Nevanlinna theory to measure complexity of tropical meromorphic functions satisfying ultra-discrete equations. They suggested, analogously to the case of difference equations in the complex plane \cite{ablowitzhh:00,halburdk:07PLMS}, that existence of sufficiently many finite-order tropical meromorphic solutions of an ultra-discrete equation is a sufficient condition for the equation in question to be of Painlev\'e type. Laine and Yang \cite{lainey:10} have laid the groundwork for the systematic study of value distribution of tropical meromorphic solutions of ultra-discrete equations by proving a number of general results applicable to large classes of ultra-discrete equations. These include a generalized ultra-discrete version of Clunie's lemma, and an analogue of Mohon'kos' lemma on value distribution of meromorphic solutions of differential equations. A study of general fundamental properties of tropical meromorphic functions has been performed by Tsai in \cite{tsai:12}. Tsai mainly discusses the family of piecewise linear functions defined on the extended real line $\mathbb{R}\cup\{-\infty\}$, and he calls tropical meromorphic functions defined on $\mathbb{R}$ by the name $\mathbb{R}$-\textit{tropical meromorphic}.


Laine and Tohge generalized tropical Nevanlinna theory to include piecewise linear functions with arbitrary real slopes, and proved a tropical version of the second main theorem for tropical meromorphic functions under a growth condition that is less restrictive than demanding finite order \cite{lainet:11}. Their results imply that behaviour of tropical meromorphic functions is in certain respects fundamentally different from their classical counterparts in the sense of value distribution. On one hand, the tropical second main theorem due to Laine and Tohge implies that under a natural non-degeneracy condition tropical meromorphic functions of finite order have no deficient values. On the other hand, a meaningful ramification term for the second main theorem in the tropical setting is yet to be discovered.



The purpose of this study is to extend tropical Nevanlinna theory to tropical holomorphic curves in a finite dimensional tropical projective space. We introduce a tropical analogue of the Cartan characteristic function for tropical holomorphic curves, and show that it reduces to the Nevanlinna characteristic due to Halburd and Southall in the one-dimensional case. As a central result of the tropical Nevanlinna-Cartan theory, we introduce a tropical analogue of Cartan's second main theorem and show that it generalizes the second main theorem by Laine and Tohge \cite{lainet:11}. In fact, our results imply a stronger version of the tropical second main theorem by Laine and Tohge in the sense that one of the conditions in their theorem can be deleted by using our tropical analogue of Cartan's second main theorem. This result, which is~Theorem \ref{tcartan2} below, implies also a second main theorem containing a tropical analogue of the ramification term expressible in terms of a tropical Casoratian.

\section{Tropical linear algebra}\label{linearsection}

In order to describe properties of tropical hyperplanes, we need to go through a number of notions from tropical linear algebra in the context of tropical entire and meromorphic functions. We start with the basic notation of tropical operations.

We define $0_{\circ}:=-\infty$ and $1_{\circ}:=0$, and we denote by $\R_{\max}$ the set $\R\cup\{-\infty\}$.
For elements $a, b\in\R_{\max}$, we define operations $\oplus $, $\odot $ and $\oslash$ by
$$
a\oplus b:=\max\{a,b\}, \quad a\odot b:=a+b \quad \text{and} \quad a\oslash b:=a-b,
$$
where $b\not=0_\circ$ in the tropical division. We also adopt the notation
$$
\frac{a}{b}\oslash := a\oslash b, \quad a^{\odot b}:= ba \quad \text{and} \quad a\odot b^{\odot (-1)}=a\oslash b.
$$
Clearly, $\max\{a,-\infty\}=\max\{-\infty, a\}=a$ and $a+(-\infty)=-\infty+a=-\infty$, for any $a\in\R_{\max}$,
so that
$$
a\oplus 0_{\circ}=0_{\circ}\oplus a = a \quad \text{and} \quad a\odot 0_{\circ}=0_{\circ}\odot a=0_{\circ}
$$
for all elements $a\in\R_{\max}$. The set $\R_{\max}$ together with the operations $\oplus $ and $\odot $, $(\R_{\max}, \oplus , \odot, 0_{\circ}, 1_{\circ})$
is called {\it max-plus algebra}, which is a semiring, that is, a non-empty set endowed with two binary operations $\oplus $ and $\odot $
such that
\smallskip
\begin{itemize}
\item $\oplus $ is associative and commutative with zero element $0_{\circ}$;
\item $\odot $ is associative, distributes over $\oplus $, and has unit element $1_{\circ}$;
\item $0_{\circ}$ is absorbing for $\odot $.
\end{itemize}
\smallskip
Since $\odot $ is commutative and $\oplus$ is idempotent, max-plus algebra is a commutative and idempotent semiring \cite{yoeli:61,speyers:09}.

The operations of addition $\oplus$ and multiplication $\odot$ for the $(n+1)\times(n+1)$ matrices $A=(a_{ij})$ and $B=(b_{ij})$ are defined by
	\begin{equation*}
	A\oplus B = (a_{ij}\oplus b_{ij})
	\end{equation*}
and
	\begin{equation*}
	A\odot B = \left(\bigoplus_{k=0}^n a_{ik}\odot b_{kj}\right),
	\end{equation*}
respectively.
An $(n+1)\times(n+1)$ matrix $A$ is called {\it regular} if $A$ contains at least one element different from $0_{\circ}$ in each row.
As in \cite{yoeli:61} we define the tropical determinant $|A|_\circ$ of $A$ by
	\begin{equation*}
	|A|_\circ = \bigoplus a_{0\pi(0)}\odot a_{1\pi(1)} \odot \cdots \odot a_{n\pi(n)},
	\end{equation*}
where the tropical summation is taken over all permutations $\{\pi(0),\pi(1),\ldots,\pi(n)\}$ of $\{0,1,\ldots,n\}$. This definition coincides with the definition of a tropical permanent, due to the fact that there is no negation in tropical arithmetic \cite{yoeli:61,maclagans:09}. The permanent of a matrix $A$ is often denoted by $\textrm{per}(A)$ or $\textrm{perm}(A)$ in the literature, see, e.g., \cite{olsderr:88}. Butkovi{\v{c}} \cite{butkovic:10} uses the term \textit{max-algebraic permanent} (or briefly permanent) and the notation $\textrm{maper}(A)$. Note that an $(n+1)\times(n+1)$ matrix $A$ is regular if and only if $|A|_\circ\neq 0_{\circ}$.

Next we will define tropical linear combinations and tropical linear independence of tropical entire and meromorphic functions. At this point it suffices to know that a tropical meromorphic function is a continuous piecewise linear function in $\mathbb{R}$, and a tropical entire function is a tropical meromorphic function with a convex graph. Clearly it does not make much sense to consider tropical linear relations of, say, tropical entire functions $g_0,\ldots,g_n$ in a directly analogous way to the classical case as
$$
\bigoplus_{\nu=0}^n a_{\nu}\odot g_{\nu}\equiv 0_\circ\,, \quad
\text{that is}, \quad \max_{0\leq \nu\leq n}\{a_{\nu}+g_{\nu}\}=-\infty,
$$
since this implies $a_{\nu}=0_\circ$ for all $\nu\in\{0, \ldots, n\}$. For the same reason linear independence in the tropical setting cannot be defined exactly analogously to the usual classical definition. There are more than one way of dealing with this issue. In the following definition we apply the notion of linear independence due to Gondran and Minoux \cite{gondranm:79,gondranm:84} for tropical meromorphic functions over the max-plus algebra $\R_{\max}$.

\begin{definition}\label{IJdef}\rm
Tropical meromorphic functions $f_0,\ldots,f_n$ are linearly dependent (respectively independent) in the \textit{Gondran-Minoux sense} if there exist (respectively there do not exist) two disjoint subsets $I$ and $J$ of $K:=\{0,\ldots,n\}$ such that $I\cup J =K$ and
    \begin{equation}\label{IJ}
    \bigoplus_{i\in I} \alpha_i \odot f_i = \bigoplus_{j\in J} \alpha_j \odot f_j,\quad \text{that~is,} \quad\max_{i\in I} \{\alpha_i + f_i\}=\max_{j\in J} \{\alpha_j +f_j\},
    \end{equation}
where the constants $\alpha_0,\ldots,\alpha_n\in\R_{\max}$ are not all equal to $0_\circ$.
\end{definition}

If either of the index sets $I$ or $J$ is empty, then the corresponding tropical sum is considered to vanish. Say that $I=\emptyset$ in \eqref{IJ}. Then $J=K$, and so \eqref{IJ} implies that $\alpha_0=\cdots=\alpha_n=0_\circ$, which is a contradiction. Therefore both $I$ and $J$ are non-empty sets in Definition \ref{IJdef}.
%
%
%
We define the notion of a tropical linear combination as follows.

\begin{definition}\label{lindef}
If $g_0,\ldots,g_n$ are tropical entire functions and $a_0,\ldots,a_n\in \R_{\max}$, then
	\begin{equation}\label{lcomb}
	f=\bigoplus_{\nu=0}^n a_{\nu}\odot g_{\nu}=\bigoplus_{i=0}^j a_{k_i}\odot g_{k_i}
	\end{equation}
is called a \textit{tropical linear combination} of $g_0,\ldots,g_n$ over $\mathbb{R}_{\max}$,
where the index set $\{k_0,\ldots,k_j\}\subset \{0,\ldots,n\}$ is such that $a_{k_i}\in\R$ for all $i\in\{0,\ldots,j\}$, while $a_{\nu}=0_{\circ}$ if $\nu\not\in\{k_0,\ldots,k_j\}$.
\end{definition}

A tropical linear combination may also be written in an inner product form $f=(a_0,\ldots,a_n)\odot (g_0,\ldots,g_n)^{\top}$ where $A^{\top}$ denotes the transpose of the matrix $A$.


If $g_0,\ldots,g_n$ do not all appear explicitly in the tropical linear combination $f$ defined by \eqref{lcomb}, then $f$ is considered to be ``degenerate'' in the sense that it does not contain explicit information from some of the functions $g_0,\ldots,g_n$.
This happens when there exists an index $\nu_0\in\{0,\ldots,n\}$ such that either the coefficient $a_{\nu_0}$ is $0_{\circ}$ or
the inequality $f(x) > a_{\nu_0}\odot g_{\nu_0}(x)$ holds for all $x\in\R$. Also, if $g_0,\ldots,g_n$ are linearly dependent in the Gondran-Minoux sense, then any tropical linear combination $f$ of these functions is degenerate in the sense that we can write $f$ in a form that does not utilize all of the functions in the set $\{g_0,\ldots,g_n\}$. We will make the notion of degeneracy precise in the rest of this section.
The first part of the following definition is an adaptation of the concepts introduced in \cite{akiangg:09}.


\begin{definition}\label{LGdef}
Let $G=\{g_0,\ldots,g_n\}(\neq\{0_{\circ}\})$ be a set of tropical entire functions, linearly independent in the Gondran-Minoux sense, and let
	$$
	\mathcal{L}_G=\textrm{span}\langle g_0,\ldots,g_n \rangle=\left\{\bigoplus_{k=0}^n a_k \odot g_k : (a_0,\ldots,a_n)\in \R_{\max}^{n+1}\right\}
	$$
be their \textit{linear span}.
(Note that $\bigoplus_{k=0}^n a_k \odot g_k=0_{\circ}$ when each $a_k$ is equal to $0_{\circ}$ so that $0_{\circ}\in \mathcal{L}_G$.)
The collection $G$ is called at the \textit{spanning basis} of $\mathcal{L}_G$. The \textit{shortest length} of the representation of $f\in \mathcal{L}_G\setminus \{0_{\circ}\}$ is defined by
	$$
	\ell(f)=\min\left\{j\in\{1,\ldots,n+1\}:f=\bigoplus_{i=1}^{j} a_{k_i} \odot g_{k_i}\right\},
	$$
where $a_{k_i}\in\R$ with integers $0\leq k_1<k_2< \cdots <k_j\leq n$, and the \textit{dimension} of $\mathcal{L}_G$ is
	\begin{equation}\label{limdef}
	\dim(\mathcal{L}_G) = \max\bigl\{\ell(f):f\in\mathcal{L}_G\setminus\{0_{\circ}\}\bigr\}.
	\end{equation}
\end{definition}

The definition of the dimension of $\mathcal{L}_G$ above implies that there exist $f\in\mathcal{L}_G\setminus\{0_{\circ}\}$ such that
    $$
    f=\bigoplus_{i=1}^{\dim(\mathcal{L}_G)}a_{k_i}\odot g_{k_i}
    $$
with $a_{k_i}\in\R$ for all $i\in\{1,\ldots,\dim(\mathcal{L}_G)\}$, but the same $f$ satisfies
    $$
    f\not=\bigoplus_{i=1}^{\dim(\mathcal{L}_G)-1}b_{m_i}\odot g_{m_i}
    $$
no matter how the collection of constants $b_{m_i}\in\R$, $i\in\{1,\ldots,\dim(\mathcal{L}_G)-1\}$, is chosen.

Now we can define completeness of tropical linear combinations in exact terms as follows.

\begin{definition}\label{completeDef}
Let $G=\{g_0,\ldots,g_n\}(\neq\{0_{\circ}\})$ be a set of tropical entire functions, linearly independent in the Gondran-Minoux sense, and let $f$ be a tropical linear combination of $g_0,\ldots,g_n$. If $\ell(f)=n+1$, then $f$ is said to be \textit{complete}. That is, the coefficients $a_k$ in any expression of $f$ of the form
    $$
    f=\bigoplus_{k=0}^na_k\odot g_k
    $$
must satisfy $a_k\in\mathbb{R}$ for all $k\in \{0, \ldots, n\}$ so that $\mathcal{L}_G$ attains its full dimension, $\dim(\mathcal{L}_G)=n+1$.
\end{definition}

Also, by Definition \ref{LGdef}, we have $\dim(\mathcal{L}_G)\geq 1$ for a non-empty finite set $G$ of tropical entire functions,
since $G\subset \mathcal{L}_G$ and $\ell(g)=1$ when $g\in G\setminus\{0_{\circ}\}$.
We illustrate the idea of defining $\dim(\mathcal{L}_G)$ in the following example.

\begin{example}\label{g012}
Let us take tropical rational functions (actually tropical polynomials) $g_0(x)=x$, $g_1(x)\equiv 1_{\circ}$ and $g_2(x)=1_\circ\oslash x$.
We see that they are linearly independent in the Gondran-Minoux sense, so that none of them can be a expressed as a tropical linear combination of the other two. The classical linear algebraic analogue of this phenomenon would mean that the dimension of the corresponding linear space is three. But now
$$
\textrm{span}\langle x,1_{\circ},-x \rangle=\bigl\{g(x:a,b,c):=\max\{x+a,b,-x+c\} : a,b,c\in \R_{\max}\bigr\}
$$
and so $\ell\bigl(g(x:a,b,c)\bigr)=3$ when and only when $a,b,c\in\R$ with $b>(a+c)/2$. Hence the maximum $3$ in \eqref{limdef} is attained for any coefficients $a,b,c\in\R$ of $g$ such that $b>(a+c)/2$.
\end{example}

Changing the spanning basis $G$ might cause a loss of information in general. For instance, if we change the triple $\{g_0, g_1,g_2\}$ fixed above in Example~\ref{g012} to
\begin{eqnarray*}
f_0(x)&:=&a_{00}\odot g_0(x)\oplus a_{01}\odot g_1(x)=\max\{x+a_{00}, a_{01}\}, \\
f_1(x)&:=&a_{11}\odot g_1(x)\oplus a_{12}\odot g_2(x)=\max\{a_{11},-x+a_{12}\}, \\
f_2(x)&:=&a_{20}\odot g_0(x)\oplus a_{22}\odot g_2(x)=\max\{x+a_{20},-x+a_{22}\}
\end{eqnarray*}
or
\begin{equation}\label{ex1eq}
\left(
\begin{array}{c}
f_0(x)\\
f_1(x)\\
f_2(x)
\end{array}
\right)
=
\left(
\begin{array}{ccc}
a_{00}    & a_{01}    & 0_{\circ} \\
0_{\circ} & a_{11}    & a_{12}    \\
a_{20}    & 0_{\circ} & a_{22}
\end{array}
\right)
\odot
\left(
\begin{array}{c}
g_0(x)\\
g_1(x)\\
g_2(x)
\end{array}
\right),
\end{equation}
then the coefficient matrix is regular in the tropical meaning provided that the tropical determinant of the coefficient matrix does not vanish. This happens, for instance, if each $a_{ij}$ in the matrix is different from $0_\circ$:
\begin{eqnarray*}
\left|
\begin{array}{ccc}
a_{00} & a_{01} & 0_\circ \\
0_\circ & a_{11} & a_{12} \\
a_{20} & 0_\circ & a_{22}
\end{array}
\right|_\circ
&=& a_{00}\odot a_{11} \odot a_{22}\oplus a_{01}\odot a_{12} \odot a_{20}\\
&=& \max\{a_{00}+ a_{11}+ a_{22}, a_{01}+ a_{12}+ a_{20}\}.
\end{eqnarray*}
The case where exactly one of the sets $\{a_{00}, a_{11}, a_{22}\}$ and $\{a_{01}, a_{12}, a_{20}\}$ contains the element $0_\circ$ is degenerate in the sense that some of the $f_j$ necessarily coincides, up to a constant, with one of the $g_j$'s, which is not very interesting from our point of view.


On the other hand, we cannot represent the function $g_1(x)=1_{\circ}$ as a linear combination of the $f_0(x)$, $f_1(x)$ and $f_2(x)$, and thus
$g_1\not\in\textrm{span}\langle f_0,f_1,f_2 \rangle$. In this sense, the regularity of a tropical matrix does not imply the same or analogous properties as in the classical case. In fact, despite being regular, the matrix in question is not invertible in the sense that there is no matrix $(b_{ij})$ $(i,j\in\{0,1,2\})$ satisfying
$$
\left(
\begin{array}{ccc}
a_{00} & a_{01} & 0_\circ \\
0_\circ & a_{11} & a_{12} \\
a_{20} & 0_\circ & a_{22}
\end{array}
\right)
\odot
\left(
\begin{array}{ccc}
b_{00} & b_{01} & b_{02} \\
b_{10} & b_{11} & b_{12} \\
b_{20} & b_{21} & b_{22}
\end{array}
\right)
=
\left(
\begin{array}{ccc}
1_\circ & 0_\circ & 0_\circ \\
0_\circ & 1_\circ & 0_\circ \\
0_\circ & 0_\circ & 1_\circ
\end{array}
\right),
$$
which can be verified by an elementary calculation. 

%

In the previous discussion the set $\{f_0,f_1,f_2\}$ includes elements that are not complete over $\{g_0,g_1,g_2\}$. We say that such a set is \textit{degenerate}. The exact definition of this notion is as follows.

\begin{definition}\label{ddgdef}
Let $G=\{g_0,\ldots,g_n\}$ be a set of tropical entire functions, linearly independent in the Gondran-Minoux sense, and let $Q\subset\mathcal{L}_G$ be a collection of tropical linear combinations of $G$ over $\mathbb{R}_{\max}$. The \textit{degree of degeneracy} of $Q$ is defined to be
	$$
	\mbox{ddg}(Q)=\card\left(\left\{f\in Q : \ell(f) < n+1\right\}\right).
	$$
If $\mbox{ddg}(Q)=0$ then $Q$ is called \textit{non-degenerate}.
\end{definition}

In other words, the degree of degeneracy of a set of tropical linear combinations is the number of its non-complete elements.

\section{Tropical meromorphic functions}\label{merosection}

Halburd and Southall \cite{halburds:09} defined a max-plus (or tropical) meromorphic function as a real continuous piecewise linear function with integer slopes, which generalizes the concept of tropical rational function (see, e.g., \cite{itenbergms:07}) in a natural way. Laine and Tohge \cite{lainet:11} showed that key results in the tropical Nevanlinna theory introduced by Halburd and Southall can be naturally extended to the case where tropical meromorphic functions have non-integer real slopes. We will extend the definitions of Halburd and Southall, and of Laine and Tohge, by introducing tropical holomorphic curves in the $n$-dimensional tropical projective space $\T\P^n=\{\R\odot \bold{a} : \bold{a}\in \R_{\max}^{n+1}\}$ in section \ref{holosection} below. Before that we need some preparatory results on properties of tropical meromorphic functions.

As in \cite{halburds:09} (see also \cite{lainet:11}), given a tropical meromorphic function $f$, we define
	$$
	\omega_f(x)=\lim_{\varepsilon\to +0}\bigl(f'(x+\varepsilon)-f'(x-\varepsilon)\bigr)
	$$
for all $x\in\R$. Note that the support of $\omega_f$ forms a discrete set with no finite limit points in $\R$.
If $\omega_f(x)<0$, then $x$ is called a \textit{pole} of $f$, while if $\omega_f(x)>0$, then $x$ is said to be a \textit{root} of $f$.
In both cases the multiplicity of the pole or the root is defined to be $|\omega_f(x)|$. We say that a real piecewise linear continuous function is \textit{tropical rational} if it has only finitely many roots and poles.
The following lemma due to Tsai characterizes tropical rational functions. It is a generalization of \cite[Lemma~2.1]{halburds:09}.

\begin{lemma}[\cite{tsai:12}, Theorem 7.3]\label{troprational}
If $f:\R\to\R$ is tropical meromorphic then it is tropical rational if and only if it can be written in the form
	\begin{equation}\label{f}
	f(t)=\left(a_0\oplus a_1 \odot t^{\odot l_1}\oplus\cdots\oplus a_p\odot t^{\odot l_p} \right)
	\oslash\left(b_0\oplus b_1\odot t^{\odot s_1}\oplus\cdots\oplus b_q\odot t^{\odot s_q}\right),
	\end{equation}
where $p,q\in\N\cup\{0\}$, $0<l_1<\cdots< l_p$,  $0<s_1<\cdots< s_q$, and the coefficients $a_0,\ldots,a_p$ and $b_0,\ldots,b_q$ are real constants.
\end{lemma}

\begin{proof}
We include an alternative proof to the one given by Tsai \cite{tsai:12}. Suppose first that a tropical meromorphic function is given by the formula \eqref{f}.
Then $f$ is a real piecewise linear continuous function, and so indeed a tropical meromorphic function.
In addition, $\omega_f(x)=0$ for all except finitely many $x\in\R$, from which it follows that $f$ has finitely many roots and poles. These roots and poles of $f$ can only appear at points where two of the terms in \eqref{f} are equal, and thus are finite in number. We conclude that $f$ is a tropical rational function.

Assume conversely that $f$ is a tropical rational function.
Then $f$ has only finitely many roots, say $\alpha_1,\ldots,\alpha_n$, and finitely many poles, say $\beta_1,\ldots,\beta_m$.
A desired representation of $f$ is then
	$$
	f(t)=\left(A_0\oplus A_1 \odot t^{\odot L_1}\oplus\cdots\oplus A_n\odot t^{\odot L_n} \right)
	\oslash\left(B_0\oplus B_1\odot t^{\odot S_1}\oplus\cdots\oplus B_m\odot t^{\odot S_m}\right),
	$$
where $A_0\in\R$ and $B_0\in\R$ are arbitrary and
	\begin{equation*}
	\begin{split}
	A_i &= A_0 - \alpha_1 \omega_f(\alpha_1)- \alpha_2 \omega_f(\alpha_2)-\cdots -\alpha_i \omega_f(\alpha_i), \\
	L_i &= \omega_f(\alpha_1)+\cdots+\omega_f(\alpha_i),\\
	B_j &= B_0 - \beta_1 |\omega_f(\beta_1)|- \beta_2 |\omega_f(\beta_2)|-\cdots -\beta_j |\omega_f(\beta_j)|, \\
	S_j &= |\omega_f(\beta_1)|+\cdots+|\omega_f(\beta_j)|,\\
	\end{split}
	\end{equation*}	
for $i\in\{1,\ldots,n\}$ and $j\in\{1,\ldots,m\}$. Note that $0<L_1<\cdots < L_n$ and $0<S_1<\cdots < S_n$.	
\end{proof}

The above representation for $f(t)$ is not determined uniquely up to two arbitrary constants $A_0\in\R$ and $B_0\in\R$, as it may seem at first glance.
In fact, besides of tropical multiplication of a {\it tropical unit}, that is, a tropical meromorphic function with neither roots nor poles and thus of the form $C_0 \odot t^{\odot M_0}$, we need to also take into account the terms in surplus in the sense that they contribute nothing to the maximum.
For example, consider the numerator of $f(t)$ in the above expression
$$
f_1(t):=A_0 \oplus A_1 \odot t^{\odot L_1} \oplus \cdots \oplus A_n \odot t^{\odot L_n}.
$$
It forms a convex hull and so we can take lines located completely below the bordering polygonal line.
Such a line is given as a graph of a monic $A_k \odot t^{\odot L_k}$ $(n+1\leq k\leq N)$ so that we have
$$
f_1(t)=A_0 \oplus A_1 \odot t^{\odot L_1} \oplus \cdots \oplus A_n \odot t^{\odot L_n}\oplus A_{n+1} \odot t^{\odot L_{n+1}}
\oplus \cdots \oplus A_{N} \odot t^{\odot L_N}.
$$
The shortest possible expression for $f_1(t)$ can be obtained by taking $g_j = t^{\odot L_j}$, $j=0,\ldots,N$, and selecting a spanning basis $G\subset\{g_0,\ldots,g_{N}\}$ such that $f_1$ is a complete tropical linear combination of $G$ over $\mathbb{R}_{\max}$ in the sense of Definition \ref{completeDef}. See also Tsai's discussion on \textit{maximally represented} polynomials \cite[section~3]{tsai:12}.

The following result is a counterpart of the so-called Borel's lemma (see \cite[Theorem A.3.3]{ru:01}) for tropical units.

\begin{proposition}
Let $f_0, \ldots,f_n$ be tropical units such that $f_i\oslash f_j$ are not constants for any distinct $i$ and $j$.
Then $f_0, \ldots,f_n$ are linearly independent in the Gondran and Minoux sense.
\end{proposition}

\begin{proof}
Suppose, on the contrary to the assertion, that there are constants $\alpha_j\in \R_{\max}$, not all equal to $0_{\circ}$, such that
    \begin{equation}\label{lindep}
    \bigoplus_{i\in I} \alpha_i \odot f_i(t) = \bigoplus_{j\in J} \alpha_j \odot f_j(t),
    \end{equation}
where $I$ and $J$ are disjoint subsets of $K:=\{0,\ldots,n\}$ such that $I\cup J = K$. By Lemma~\ref{troprational} both the left and the right side of \eqref{lindep} are tropical rational, and hence there exist $t_0\in\R$, $i_0\in I$ and $j_0\in J$ such that
    $$
    \alpha_{i_0} \odot f_{i_0}(t) = \alpha_{j_0} \odot f_{j_0}(t)
    $$
for all $t\geq t_0$. But since $f_{i_0}$ and $f_{j_0}$ are tropical units, it follows that
    $$
    f_{i_0}\oslash f_{j_0} = \alpha_{j_0} \oslash \alpha_{i_0},
    $$
which is a contradiction.
\end{proof}

If a tropical meromorphic function does not have any poles, we say it is \textit{tropical entire}.
All meromorphic functions in the complex plane can be represented as a quotient of two entire functions, which do not have any common roots.
A parallel result is valid also in the tropical real line.

\begin{proposition}\label{fgh}
For any tropical meromorphic function $f$ there exist tropical entire functions $g$ and $h$ such that $f=h\oslash g$, where $g$ and $h$ do not have any common roots.
\end{proposition}

\begin{proof}
Let $n\in\N$, and let $f_n$ be the restriction of $f$ to the interval $I_n=[-n,n]$, i.e., $f_n:I_n\to\R$ such that $f_n(t)=f(t)$ for all $t\in I_n$. The fact that roots and poles of a tropical meromorphic function have no finite limit points implies that $f_n$ has only finitely many of them, and so it follows that there exists a tropical rational function $R_n$ such that $R_n(t)=f_n(t)$ for all $t\in I_n$ and $R_n$ does not have any roots or poles outside of $I_n$. By Lemma~\ref{troprational} $R_n$ can be represented in the form
	\begin{equation*}
	R_n(t)=P_n(t) \oslash Q_n(t),
	\end{equation*}
where $P_n(t)$ and $Q_n(t)$ are tropical entire functions with finitely many roots (i.e. tropical polynomials). Also, $P_n$ and $Q_n$ do not have any roots in common, and we can take them such that
    $$
    P_{n+1}(t)\equiv P_n(t) \quad \textrm{and}\quad Q_{n+1}(t)\equiv Q_n(t)
    $$
on the interval $I_n$, and therefore
    $$
    h(t):= \lim_{k\to\infty} P_k(t) = P_n(t) \quad \textrm{and}\quad g(t):= \lim_{k\to\infty} Q_k(t) = Q_n(t)
    $$
for all $t\in I_n$. Hence the assertion follows by taking $h=\lim_{n\to\infty} P_n$ and $g=\lim_{n\to\infty} Q_n$, where the convergence is locally uniform.
\end{proof}

We now recall the definitions of tropical Nevanlinna functions from \cite{halburds:09,lainet:11}. The tropical proximity function of a tropical meromorphic function $f(x)=h(x)\oslash g(x)$ is
	$$
	m(r,f)=\frac{1}{2}\bigl(f^+(r)+f^+(-r)\bigr),
	$$
where
	$$
	f^+(x)=\max\{0,f(x)\},
	$$
and the tropical counting function is
	$$
	N(r,f)=\frac{1}{2}\int_0^r n(t,f)dt,
	$$
where
	$$
	n(x,f)=\sum_{|s|\leq x \atop \omega_f(s) < 0 } |\omega_f(s)|.
	$$
The tropical Nevanlinna characteristic function is then defined in a usual way as
	$$
	T(r,f)=m(r,f)+N(r,f),
	$$
and it satisfies
	\begin{equation}\label{tPJ}
	T(r,f)-T(r,1_\circ\oslash f)=f(0)
	\end{equation}
which is the tropical Jensen formula \cite{halburds:09,lainet:11}. Now the \textit{hyper-order} of a tropical meromorphic function $f$ can be defined as
	\begin{equation}\label{hyperorder}
	\varsigma(f)=\limsup_{r\to\infty}\frac{\log\log T(r,f)}{\log r}.
	\end{equation}
Tropical Nevanlinna functions satisfy similar basic inequalities as their classical counterparts. For instance, if $f_1,\ldots,f_n$ are tropical meromorphic functions, then
    \begin{equation}\label{proximityineq1}
    m\left(r,\bigoplus_{j=1}^n f_j\right) \leq \sum_{j=1}^n m\left(r,f_j\right)
    \end{equation}
and
    \begin{equation}\label{proximityineq2}
    m\left(r,\bigodot_{j=1}^n f_j\right) \leq \bigodot_{j=1}^n m\left(r,f_j\right).
    \end{equation}
Similar inequalities hold also for the counting function and the characteristic function.

Halburd and Southall obtained a tropical analogue of the lemma on the logarithmic derivative for finite-order tropical meromorphic functions in \cite{halburds:09}. Their result was extended to the case of hyper-order strictly less than one by Laine and Tohge \cite{lainet:11}.

\begin{theorem}[\cite{lainet:11}]\label{logder}
If $\varepsilon>0$, $c\in\R$ and $f$ is a tropical meromorphic function such that $\varsigma(f)=\varsigma<1$, then
	\begin{equation}\label{logderanalogue}
	m(r,f(x+c)\oslash f(x))=o\left(\frac{T(r,f)}{r^{1-\varsigma-\varepsilon}}\right)
	\end{equation}
as $r$ approaches infinity outside of a set of finite logarithmic measure.
\end{theorem}

It was confirmed in \cite{lainet:11} by introducing a suitable tropical meromorphic function of hyper-order one that the assumption $\varsigma<1$ in Theorem \ref{logder} is sharp.

The following lemma on growth properties of non-decreasing continuous real functions is from \cite{halburdkt:09} (see also \cite[Lemma 2.1]{halburdk:07PLMS}).

\begin{lemma}[\cite{halburdkt:09}]\label{technical}
Let $T:[0,+\infty)\to[0,+\infty)$ be a non-decreasing continuous
function and let $s\in(0,\infty)$. If the hyper-order of $T$ is
strictly less than one, i.e.,
    \begin{equation}\label{assu}
    \limsup_{r\to\infty}\frac{\log\log T(r)}{\log r}=\varsigma<1
    \end{equation}
and $\delta\in(0,1-\varsigma)$ then
   \begin{equation}\label{concl}
    T(r+s) = T(r)+ o\left(\frac{T(r)}{r^{\delta}}\right)
    \end{equation}
where $r$ runs to infinity outside of a set of finite logarithmic
measure.
\end{lemma}

By applying Theorem \ref{logder} with $c=\xi-\eta$, substituting $x\to x+\eta$ into \eqref{logderanalogue} and using Lemma~\ref{technical}, we obtain the following consequence of Theorem \ref{logder}.

\begin{corollary}\label{logdercor}
If $\varepsilon>0$, $\xi,\eta\in\R$ and $f$ is a tropical meromorphic function such that $\varsigma(f)=\varsigma<1$, then
	\begin{equation*}
	m(r,f(x+\xi)\oslash f(x+\eta))=o\left(\frac{T(r,f)}{r^{1-\varsigma-\varepsilon}}\right)
	\end{equation*}
as $r$ approaches to infinity outside of a set of finite logarithmic measure.
\end{corollary}

\section{Tropical holomorphic curves}\label{holosection}

In this section we extend some of key notions and results of section \ref{merosection} to the tropical projective space $\T\P^n$. The space $\T\P^n$ is given as a quotient space of  $\R_{\max}^{n+1}\setminus\{\mathbf{0}_\circ\}$ by equivalence relation~$\sim$, where $\mathbf{0}_\circ=(0_\circ,\ldots,0_\circ)$ is the zero element of $\R_{\max}^{n+1}$, and
    $$
    (a_0, a_1, \ldots , a_n) \sim (b_0, b_1, \ldots , b_n)
    $$
if and only if
    $$
    (a_0, a_1, \ldots , a_n) = \lambda \odot (b_0, b_1, \ldots , b_n):= (\lambda \odot b_0,\lambda \odot  b_1, \ldots ,\lambda \odot  b_n)
    $$
for some $\lambda \in \Bbb{R}$. We denote by $[a_0: a_1: \cdots : a_n]$ the equivalence class of $(a_0, a_1, \ldots , a_n)$. When $a_0\in\R$, we may take $(a_1\oslash a_0, \ldots , a_n \oslash a_0)\in \R_{\max}^n$ as a representative element of $[a_0: a_1: \cdots : a_n]$.  For example, $\Bbb{TP}^{1}$ is identical to the {\it completed} max-plus semiring $\R_{\max}\cup\{+\infty\}=\R\cup\{\pm \infty\}$ by the map such that
$$
[1_{\circ}:a] \mapsto a\oslash 1_{\circ} = a \text{ for } a\in\R_{\max}\,,
$$
and
$$
[0_{\circ}:a] \mapsto a \oslash 0_{\circ} =+\infty  \text{ for } a\in\R.
$$

Proposition \ref{fgh} implies that any tropical meromorphic function $f$ can be represented in the form $f=[g:h]$, where $g$ and $h$ are tropical entire functions without common roots. This concept is naturally generalized as follows.

\begin{definition}
Let
$$
f=[g_0:\cdots:g_n]:\R\to\T\P^n
$$
be a tropical {\it holomorphic} map where $g_0,\ldots,g_n$ are tropical entire functions and do not have any roots which are common to all of them.
Denote $\mathbf{f}=(g_0,\ldots,g_n):\R\to\R^{n+1}$.
Then the map $\mathbf{f}$ is called a \textit{reduced representation} of the {\it tropical holomorphic curve} $f$ in $\T\P^n$.
\end{definition}

We will now introduce a Cartan characteristic function for tropical holomorphic curves. The definition is remarkably simple.

\begin{definition}\label{dfrdef}
If $f:\R\to\T\P^n$ is a tropical holomorphic curve with a reduced representation $\mathbf{f}=(g_0,\ldots,g_n)$, then
	$$
	T_{\mathbf{f}}(r)=\frac{1}{2}\left(F(r)+F(-r)\right)-F(0), \qquad F(x)=\max\{g_0(x),\ldots,g_n(x)\},
	$$
is said to be the \textit{tropical Cartan characteristic function} of $f$.
\end{definition}

Despite the apparent simplicity of its definition, the tropical Cartan characteristic function carries all the information held in the usual tropical Nevanlinna characteristic. In order to make sure that $T_{\mathbf{f}}(r)$ is a well defined characteristic function, we first show that it is independent of the reduced representation $\mathbf{f}$ of $f$.

\begin{proposition}\label{reducedindep}
The tropical Cartan characteristic function $T_{\mathbf{f}}(r)$ is independent of the reduced representation of the tropical holomorphic curve $f$.
\end{proposition}

\begin{proof}
Take any two reduced representations
$$
\mathbf{f}:=(g_0,g_1, \ldots , g_n) \quad \text{ and } \quad
\widetilde{\mathbf{f}}:=(\tilde{g}_0, \tilde{g}_1, \ldots , \tilde{g}_n)
$$
of a given tropical holomorphic curve $f:\R\to\T\P^{n}$.
By definition (of projective coordinates), it follows directly that
for each point $x\in\R$, there is $\lambda\in\R$ depending on $x$,
that is, $\lambda=\lambda(x)$, such that
$$
\tilde{g}_j(x)=g_j(x)+\lambda(x)
$$
holds for $0\leq j \leq n$.
Recall that $g_0(x),\ldots,g_n(x)$ is a set of $n+1$ tropical entire functions
with no common roots, and the same is true for $\tilde{g}_0(x),\ldots,\tilde{g}_n(x)$.
Hence, for all $j\in\{0,\ldots,n\}$, the function $\lambda(x)\equiv \tilde{g}_j(x)-g_j(x)$ is tropical meromorphic such that
$\omega_{\lambda}(x)\equiv \omega_{\tilde{g}_j}(x)-\omega_{g_j}(x)$.
If $\lambda(x)$ had either a root or a pole, that point would be a common root of
the $\tilde{g}_j(x)$ or the $g_j(x)$, respectively, which is a contradiction.
Hence $\lambda(x)$ can have neither roots nor poles, so that it is
a linear function, say $\lambda(x)=\alpha x+\beta$ on $\R$.

We have so far shown that $\tilde g_j(x)=g_j(x)+\alpha x + \beta$, where $\alpha$ and $\beta$ are real constants that do not depend on
$j\in\{0,\ldots,n\}$ or $x\in\R$, and certainly not on $r=|x|$. Therefore, we have
	$$
	\tilde F(x)=\max\{\tilde g_0(x),\ldots,\tilde g_n(x)\} = F(x)+\alpha x + \beta,
	$$
and
    $$
    \tilde{F}(0)=F(0)+\beta.
    $$
Hence,
	$$
	\frac{1}{2}\Bigl(\tilde F(r)+\tilde F(-r)\Bigr)-\tilde{F}(0)=\frac{1}{2}\Bigl(F(r)+F(-r)\Bigr) -F(0),
	$$
which means that the characteristic functions $T_{\mathbf{f}}(r)$ and $T_{\tilde{\mathbf{f}}}(r)$ of $f$ coincide.
\end{proof}


By Proposition \ref{reducedindep} the tropical Cartan characteristic function is independent of the reduced representation $\mathbf{f}$ of the tropical holomorphic curve $f$. For this reason we will adopt the notation $T_f(r)$ instead of $T_{\mathbf{f}}(r)$ from now on. The following proposition shows that $T_{f}(r)$ coincides, up to a constant, with the Nevanlinna characteristic introduced in \cite{halburds:09,lainet:11} when $f=[g:h]:\R\to\T\P^1$.

\begin{proposition}\label{consistency}
If $f=h\oslash g$ is a tropical meromorphic function such that $f(0)=1_\circ=0$, then
	$$
	T_{\mathbf{f}}(r)=T_{f}(r)=T(r,f),	
	$$
where $\mathbf{f}=(g,h)$ is a reduced representation of the holomorphic curve $f:\R\to\T\P^1$ given by $f=[g:h]$.
\end{proposition}

\begin{proof}
The function $F(x)$ in the definition of $T_{\mathbf{f}}(r)=T_f(r)$ can be written in the form
	\begin{equation}\label{F}
	F(x) = (h-g)^+(x)+g(x)
	\end{equation}
for all $x \in \R$. By applying \eqref{F} with $x=\pm r$, it follows that
	\begin{equation}\label{Tf}
	T_{f}(r)=m(r,h-g)+\frac{1}{2}g(r)+\frac{1}{2}g(-r)-(h-g)^{+}(0)-g(0).
	\end{equation}
Since
	$$
	g(x)=g^+(x)-(-g)^+(x)
	$$	
for all $x\in\R$, it follows, in particular, that
	$$
	g(r)+g(-r)=g^+(r)+g^+(-r)-(-g)^+(r) -(-g)^+(-r),
	$$	
and so
	\begin{equation}\label{fromtPJ}
	\frac{1}{2}g(r)+\frac{1}{2}g(-r)=m(r,g)-m(r,1_\circ\oslash g).
	\end{equation}
Moreover, by the tropical Jensen formula \eqref{tPJ}, we have
	\begin{equation}\label{fromtPJ2}
	m(r,g)-m(r,1_\circ\oslash g)=N(r,1_\circ\oslash g)-N(r,g)+g(0).
	\end{equation}
By combining \eqref{Tf} with \eqref{fromtPJ} and \eqref{fromtPJ2} it follows that
	\begin{equation}\label{Tf2}
	T_{f}(r)=m(r,h \oslash g)-N(r,g)+N(r,1_\circ\oslash g)-(h-g)^{+}(0).
	\end{equation}
The counting function $N(r,g)$ vanishes identically since $g$ is entire. Furthermore, since the tropical entire functions $g$ and $h$ do not have any common roots, it follows that
    $$
    N(r,1_\circ\oslash g)=N(r,h\oslash g).
    $$
Hence \eqref{Tf2} becomes the desired formula
	$$
	T_{f}(r)=T(r,h\oslash g)-(h\oslash g)^{+}(0)=T(r,f)-f^{+}(0).
	$$
\end{proof}

\begin{remark} By the normalization $f(0)=1_\circ$, we have just obtained
    $$
    T_{f}(r)=T(r,f),
    $$
while without the normalization, we have
    $$
    T_{f}(r)=T(r,f)-f^{+}(0).
    $$
\end{remark}

Proposition \ref{consistency} shows that the characteristic functions $T_{f}(r)$ and $T(r,f)$ are equal up to a constant if $f:\R\to\T\P^1$. Hence in the one-dimensional case the hyper-order \eqref{hyperorder}, for instance, may be defined using $T_{f}(r)$ instead of $T(r,f)$. In the case when $f:\R\to\T\P^n$ for $n\geq 2$ we can still obtain a useful inequality in terms of tropical linear combinations of the coordinates of $f$. In order to state this inequality we introduce a new counting function of the common roots of two given tropical entire functions.

\begin{definition}
Let $h_1$ and $h_2$ be tropical entire functions. We define
$$
N_{\min}\left(r,0,h_1,h_2\right)
=\min\bigl\{N(r, 1_{\circ}\oslash h_1), N(r, 1_{\circ}\oslash h_2)\bigr\}.
$$
\end{definition}
%
%

\begin{lemma}\label{Tlemma}
If $g=[g_0:\cdots:g_n]:\R\to\T\P^n$ is a tropical holomorphic curve, then
	$$
	T\left(r,\hat f \oslash \tilde f \right)+N_{\min}\left(r,0,\hat{f},\tilde{f}\right) \leq T_{g}(r)+C+\max\{g_0(0),\ldots,g_n(0)\}-\tilde{f}(0),
    $$
where $\hat f$ and $\tilde f$ are linear combinations of the $n+1$ tropical entire functions $g_0,\ldots,g_n$ without common roots and $C$ is the maximum of the coefficients of these two linear combinations over $\R_{\max}$.
\end{lemma}

\begin{proof}
Let $u$ be a tropical entire function that satisfies
	\begin{equation}\label{omegau}
	\omega_u(x)=\min\{\omega_{\hat f}(x),\omega_{\tilde f}(x)\}
	\end{equation}
for all $x\in\R$, and also $u(0)=\tilde{f}(0)$. Let $\hat w := \hat f \oslash u$ and $\tilde w := \tilde f \oslash u$. Since $\omega_{\hat w}=\omega_{\hat f}-\omega_{u}\geq 0$ and $\omega_{\tilde w}=\omega_{\tilde f}-\omega_{u}\geq 0$ by definition, it follows that $\hat w$ and $\tilde w$ are tropical entire functions such that $\tilde{w}(0)=\tilde{f}(0)-u(0)=1_\circ$. If $\hat w$ and $\tilde w$ have a common root, say $x_0$, then $\omega_{\hat w}(x_0)>0$ and $\omega_{\tilde w}(x_0)>0$. But this implies that $\omega_{\hat f}(x_0)>\omega_u(x_0)$ and $\omega_{\tilde f}(x_0)>\omega_u(x_0)$, which contradicts \eqref{omegau}. We have therefore shown that there exist tropical entire functions $\hat w$, $\tilde w$ and $u$ such that  $\hat f  = u \odot \hat w$ and $\tilde f  = u \odot \tilde w$, and such that $\hat w$ and $\tilde w$ do not have any common roots.

By Proposition \ref{consistency} it follows that
	\begin{equation}\label{key1}
	T\left(r,\hat f \oslash \tilde f \right) = T\left(r,\hat w \oslash \tilde w \right) = \frac{1}{2}\left(G(r)+G(-r)\right)-G(0),
	\end{equation}
where $G(x)=\max\{\hat w(x),\tilde w(x)\}$. Note that $G(0)\geq 0$. Since $\hat f$ and $\tilde f$ are tropical linear combinations of $g_0,\ldots,g_n$, it follows that
	\begin{equation}\label{key2}
	\max\{\hat f(x),\tilde f(x)\} \leq F(x)+C,
	\end{equation}
where $F(x)=\max\{g_0(x),\ldots,g_n(x)\}$. Therefore, by combining Definition \ref{dfrdef} with equations \eqref{key1} and \eqref{key2}, we have
	\begin{equation}\label{key3}
	T\left(r,\hat f \oslash \tilde f \right) \leq T_{g}(r)+F(0)- \frac{1}{2}\left(u(r)+u(-r)\right)+C.
	\end{equation}
The tropical Jensen formula \eqref{tPJ} yields
	\begin{equation*}
	\frac{1}{2}\left(u(r)+u(-r)\right) = m(r,u)-m(r,1_\circ\oslash u)=N(r,1_\circ\oslash u)-N(r,u)+u(0),
	\end{equation*}
where the counting function $N(r,u)$ vanishes identically due to the fact that $u$ is entire. Therefore \eqref{key3} becomes
	\begin{equation*}
	T\left(r,\hat f \oslash \tilde f \right) \leq T_{g}(r)-N(r,1_\circ\oslash u)+C+\max\{g_0(0),\ldots,g_n(0)\}-\tilde{f}(0),
	\end{equation*}
which implies the assertion.
\end{proof}

\section{Tropical Casoratian}

The tropical Casorati determinant plays the role of the Wronskian in the tropical analogue of Cartan's second main theorem in section \ref{cartansection} below. In this short section we describe some of the basic properties of the tropical Casoratian.

Let $g(x)$ be a tropical entire function, and choose $c\in\R\setminus\{0\}$ that will be fixed from now on. In fact, using the transformation $x\mapsto cx$ we can take $c=1$ in the following considerations without any loss of generality. We denote, for $n\in\N$,
    \begin{equation*}
    g(x)\equiv g, \quad g(x+c)\equiv \overline{g},\quad g(x+2c)\equiv
    \overline{\overline{g}}\quad \textrm{and} \quad g(x+nc)\equiv
    \overline{g}^{[n]}.
    \end{equation*}
The \textit{tropical Casorati determinant}, or \textit{tropical Casoratian}, of tropical entire functions $g_0,\ldots,g_n$ is defined by
	\begin{equation}\label{C1}
	C_\circ(g_0,g_1,\ldots,g_n) = \bigoplus \overline{g}_0^{[\pi(0)]} \odot \overline{g}_1^{[\pi(1)]}  \odot \cdots \odot \overline{g}_n^{[\pi(n)]},
	\end{equation}
where the sum is taken over all permutations $\{\pi(0),\ldots,\pi(n)\}$ of $\{0,\ldots,n\}$. It satisfies the following simple properties.

\begin{lemma}\label{casoratilemma}
If $g_0,\ldots,g_n$ and $h$ are tropical entire functions, then
	\begin{enumerate}
	\item[(i)] $C_\circ(g_0,g_1,\ldots,g_i,\ldots,g_j,\ldots,g_n)=C_\circ(g_0,g_1,\ldots,g_j,\ldots,g_i,\ldots,g_n)$ for all $i,j\in\{0,\ldots,n\}$ such that $i\not=j$.
	\item[(ii)] $C_\circ(1_\circ,g_1,\ldots,g_n)\geq C_\circ(\overline{g}_1,\ldots,\overline{g}_n)$.
	\item[(iii)] $C_\circ(0_\circ,g_1,\ldots,g_n)=0_\circ$.
	\item[(iv)] $C_\circ(g_0\odot h,g_1 \odot h,\ldots,g_n\odot h)=h\odot\overline{h}\odot\cdots\odot \overline{h}^{[n]}\odot C_\circ(g_0,g_1,\ldots,g_n)$.
	\end{enumerate}
\end{lemma}

\begin{proof}
Property (i) is obtained by changing the order of summation in the determinant, and property (iii) follows by substituting $0_\circ=-\infty$ to the definition of the tropical Casoratian.

To verify property (ii), we expand $C_{\circ}(g_0, g_1, \ldots , g_n)$ according to the first column vector to obtain
$$
C_{\circ}(g_0, g_1,\ldots, g_n)
=g_0\odot C_{\circ}(\overline{g}_1,\ldots,\overline{g}_n)\oplus \cdots \oplus \overline{g}^{[n]}_0\odot C_{\circ}(g_1,,\ldots, g_n),
$$
from which the assertion follows by taking $g_0=1_\circ$.

To show that (iv) is valid, we observe that since the tropical determinant is invariant under transposing, it follows from \eqref{C1} that
	\begin{equation}\label{C2}
	C_\circ(g_0,g_1,\ldots,g_n) = \bigoplus \overline{g}_{\pi(0)}^{[0]} \odot \overline{g}_{\pi(1)}^{[1]}  \odot \cdots \odot \overline{g}_{\pi(n)}^{[n]}.
	\end{equation}
By a tropical side by side multiplication of \eqref{C2} with $h\odot \overline{h}\odot\cdots\odot\overline{h}^{[n]}$, we have
	\begin{equation*}
	\begin{split}
	h\odot \overline{h} & \odot\cdots\odot\overline{h}^{[n]}\odot C_\circ(g_0,g_1,\ldots,g_n) \\& = \bigoplus (\overline{g_{\pi(0)}\odot h})^{[0]} \odot (\overline{g_{\pi(1)}\odot h})^{[1]}  \odot \cdots \odot (\overline{g_{\pi(n)}\odot h})^{[n]}\\&= \bigoplus (\overline{g_{0}\odot h})^{[\pi(0)]} \odot (\overline{g_{1}\odot h})^{[\pi(1)]}  \odot \cdots \odot (\overline{g_{n}\odot h})^{[\pi(n)]}\\ &=C_\circ(g_0\odot h,g_1 \odot h,\ldots,g_n\odot h).
	\end{split}
	\end{equation*}
\end{proof}

\begin{remark}
The properties (i)-(iv) of Lemma \ref{casoratilemma} are possessed by the ordinary Casoratian $C(g_0, g_1, \ldots , g_n)$ in the following similar form:
\smallskip
    \begin{enumerate}
    \item $C(g_0, g_1, \ldots , g_i, \ldots , g_j, \ldots, g_n) = - C(g_0, g_1, \ldots , g_j, \ldots , g_i, \ldots, g_n)$ for all $i,j \in \{0, \ldots , n\}$.
    \item $C(1, g_1, \ldots , g_n)=C(\Delta g_1, \ldots , \Delta g_n)$ with $\Delta g= \overline{g} - g$.
    \item $C(0, g_1, \ldots , g_n)=0$.
    \item $C(h g_0, h g_1,  \ldots , h g_n)(x)=\prod_{k=0}^{n} h(x + k) C(g_0, g_1,  \ldots , g_n)(x)$.
    \end{enumerate}
\smallskip
Here $g_0, g_1, \ldots , g_n$ and $h$ are often assumed to be entire or meromorphic functions in the complex plane, for instance.
\end{remark}

Concerning Lemma \ref{casoratilemma} (ii), an opposite inequality does not appear to hold in general.
At each $x\in \mathbb{R}$, there exists a permutation $\pi_x(j)$ of $\{0,1, \ldots , n\}$ such that
$$
C_{\circ}(1_{\circ}, g_1, \ldots , g_n)(x)=\overline{1_{\circ}}^{[\pi_x(0)]}(x)\odot\overline{g}_1^{[\pi_x(1)]}(x) \odot \cdots \odot\overline{g}_n^{[\pi_x(n)]}.
$$
If $\pi_x(0)=0$, then $\pi_x(j)$ $(1\leq j \leq n)$ is a permutation of $\{1,\ldots , n\}$ and therefore we have
$$
C_{\circ}(1_{\circ}, g_1, \ldots , g_n)(x)
=\overline{g}_1^{[\pi_x(1)]}(x)\odot \cdots \odot \overline{g}_n^{[\pi_x(n)]}
\leq C_{\circ}(\overline{g}_1, \ldots ,\overline{g}_n)(x),
$$
while if $\pi(0)=n$, $\pi_x(j+1)$ $(0\leq j \leq n-1)$ is a permutation of $\{0,\ldots , n-1\}$ and therefore we have
$$
C_{\circ}(1_{\circ}, g_1, \ldots , g_n)(x)
=\overline{g}_1^{[\pi_x(1)]}(x)\odot \cdots \odot \overline{g}_n^{[\pi_x(n)]}
\leq C_{\circ}(g_1, \ldots , g_n)(x).
$$
But otherwise, it is unlikely that any similar expression to the two cases can be found unless we restrict the forms of $g_1,\ldots,g_n$.

\section{Tropical version of Cartan's second main theorem}\label{cartansection}

In this section we present the main result of this study, which is a tropical analogue of Cartan's second main theorem for holomorphic curves in $\P^n(\C)$. Cartan's original result \cite{cartan:33} from 1933 is a natural generalization of Nevanlinna's second main theorem for meromorphic functions in $\C$. At the end of this section we show how our tropical version of Cartan's second main theorem implies a tropical second main theorem due to Laine and Tohge. First, we recall the second main theorem by Cartan in the following form.

\begin{theorem}[Cartan \cite{cartan:33}]\label{classicalcartan}
Let $f:\C\to\P^n(\C)$ be a linearly non-degenerate holomorphic curve. Let $H_j$, $j=0,\ldots,q$, be $q+1$ hyperplanes of $\P^n(\C)$ in general position. Then
    $$
    (q-n)T_f(r) \leq \sum_{j=0}^q N_f(r,H_j) - N_W(r,0) + O(\log^+(rT_f(r)))
    $$
as $r\to\infty$ outside of a set of finite linear measure.
\end{theorem}

Here $N_W(r,0)$ is a ramification term defined in terms of the Wronskian determinant of the entire component functions of $f$ and $N_f(r,H_j)$ is the counting function of $f$ for the hyperplane $H_j$, while $T_f(r)$ is the {\it classical} Cartan characteristic function of $f$ to be defined below. Cartan conjectured that if the image of a holomorphic curve in $\P^n(\C)$ spans a linear subspace of codimension $s$, then the inequality
    $$
    (q-n-s)T_f(r) \leq \sum_{j=0}^q N_f(r,H_j) - \frac{n+1}{n+1-s}N_W(r,0) + O(\log^+(rT_f(r)))
    $$
holds nearly everywhere. Cartan's conjecture was finally proved, almost half a century after Cartan's original formulation, by Nochka \cite{nochka:83} in 1983.

%


The following two theorems are tropical analogues of Nochka's extension of Cartan's second main theorem. The first one is formulated with an explicit form of the error term (see \cite{cherryy:01} for an account on sharp forms of the error terms in the classical value distribution theory) and it holds for all tropical holomorphic curves without a growth condition.

\begin{theorem}\label{tcartan}
Let $q$ and $n$ be positive integers with $q>n$, and let $\varepsilon>0$.
Given $n+1$ tropical entire functions $g_0, \ldots , g_n$ without common roots, and linearly independent in Gondran-Minoux sense, let the $q+1$ tropical linear combinations $f_{0}, \ldots , f_q$ of the $g_j$ over the semi-ring $\mathbb{R}_{\max}$ be defined by
$$
f_k(x)=a_{0k}\odot g_0(x)\oplus a_{1k}\odot g_1(x)\odot \cdots \odot a_{nk}\odot g_n(x), \quad 0 \leq k \leq q.
$$
Let $
    \lambda=\mbox{ddg}(\{f_{n+1}, \ldots , f_q\})
    $
and
    \begin{equation}\label{tildeL}
	\tilde L=\frac{f_0\odot \overline{f}_1 \odot \cdots \odot \overline{f}_n^{[n]}\odot f_{n+1}\odot\cdots\odot f_q}{\ct(f_0,f_1,\ldots,f_n)}\oslash.
	\end{equation}
If the tropical holomorphic curve $g$ of $\mathbb{R}$ into $\mathbb{TP}^{n}$ has a reduced representation $\mathbf{g}=(g_0, \ldots , g_n)$,
then
\begin{equation}\label{tcartanineq}
\begin{split}
(q-n-\lambda)T_{g}(r) &\leq N\bigl(r, 1_{\circ}\oslash \tilde L\bigr)
-N(r, \tilde L) \\ &\quad  + \sum_{j=1}^q\sum_{l=0}^n \sum_{m=0}^n m\left(r,(\overline{f}_{j}^{[l]}\oslash\overline{f}_0^{[l]}) \oslash (\overline{f}_j^{[m]} \oslash \overline{f}_0^{[m]})\right)\\
 &\quad + L(0) + \sum_{j=1}^n \left( \max_{0 \leq i \leq n} \{a_{ij}+g_i(j)\}- \max_{0 \leq i \leq n} \{a_{ij}+g_i(0)\} \right) \\
    &\quad - \sum_{\nu=n+1 \atop \card(I_{\nu})=n+1}^q \min_{j\in I_{\nu}} \{ a_{j\nu}\}  -\sum_{\nu=n+1 \atop \card(I_{\nu})< n+1}^q \max_{j \in I_{\nu}} \{ a_{j\nu} + g_j(0) \} \\ &\quad - (q-n-\lambda)\max_{j\in\{0,\ldots,n\}}\{g_j(0)\},
\end{split}
\end{equation}
where the set $I_\nu$ consists of all indices $j\in \{0,\ldots,n\}$ such that $a_{j\nu}\not=-\infty$, and
    \begin{equation}\label{L}
    L=\frac{f_0\odot f_1 \odot \cdots \odot f_{n}\odot f_{n+1} \odot \cdots \odot f_{q}}{C_{\circ}(f_0, f_1, \cdots , f_n)}\oslash.
    \end{equation}
\end{theorem}

The reason why we have introduced the function $L$ in \eqref{L} and in the error term of \eqref{tcartanineq} is that it, rather than $\tilde L$, leads to a more natural tropical analogue of Cartan's second main theorem, as can be seen by a comparison with Theorem~\ref{classicalcartan}. In Theorem~\ref{tcartan2} below we will be able to replace $\tilde L$ with $L$ completely in \eqref{tcartanineq}, by imposing a growth condition on $g$.

\begin{proof}[Proof of Theorem \ref{tcartan}.] Lemma \ref{casoratilemma} yields
	\begin{equation*}
	\ct(f_0,f_1,\ldots,f_n)= f_0 \odot \overline{f}_0 \odot \cdots \odot \overline{f}_0^{[n]} \odot \ct(1_\circ,f_1 \oslash f_0,\ldots,f_n \oslash f_0),
	\end{equation*}
and so, by defining
    \begin{equation*}
	\tilde L=\frac{f_0\odot \overline{f}_1 \odot \cdots \odot \overline{f}_n^{[n]}\odot f_{n+1}\odot\cdots\odot f_q}{\ct(f_0,f_1,\ldots,f_n)}\oslash,
	\end{equation*}
and
	\begin{equation*}
	v=f_{n+1} \odot \cdots \odot f_q,
	\end{equation*}
it follows that
	\begin{equation}\label{vLK}
	v=\tilde L\odot K,
	\end{equation}
where
	\begin{equation}\label{K}
	K=\ct(1_\circ,f_1 \oslash f_0,\ldots,f_n \oslash f_0)  \odot  (\overline{f}_0\oslash\overline{f}_1) \odot \cdots \odot (\overline{f}_0^{[n]}\oslash\overline{f}_n^{[n]}).
	\end{equation}

We will now show that
	\begin{equation}\label{v1}
    \begin{split}
	(q-n-\lambda)T_g(r) &\leq \frac12 v(r)+\frac12 v(-r) -\sum_{\nu=n+1 \atop \card(I_{\nu})=n+1}^q \min_{j\in I_{\nu}} \{a_{j\nu}\} \\ & \quad -\sum_{\nu=n+1 \atop \card(I_{\nu})< n+1}^q \max_{j \in I_{\nu}} \{ a_{j\nu} + g_j(0) \} - (q-n-\lambda)\max_{j\in\{0,\ldots,n\}}\{g_j(0)\}
    \end{split}
	\end{equation}
for all $r$ sufficiently large. This estimate is verified in the following way. We represent the $f_{\nu}$ $(n+1 \leq \nu \leq q)$ by
    \begin{equation}\label{f01}
    f_{\nu}(x) = \max_{j \in I_{\nu}} \{ a_{j\nu} + g_j(x) \}, \quad a_{j\nu} \in \mathbb{R},
    \end{equation}
for index sets $I_{\nu} \subset \{0, 1,  \ldots ,  n\}$ of cardinality $\card(I_{\nu})$. Note that $f_\nu(x)$ is itself a tropical entire function since it is a tropical linear combination of tropical entire functions. Since $f_\nu(x)$ is piecewise linear function, there exists $\alpha_\nu,\beta_\nu\in\R$ and an interval $[r_1,r_2]\subset \R$ containing the origin such that $r_1<r_2$ and
    \begin{equation}\label{f02}
    f_\nu(x) = \alpha_\nu x + \beta_\nu
    \end{equation}
for all $x\in [r_1,r_2]$. Since $0\in[r_1,r_2]$, we have by combining \eqref{f01} and \eqref{f02} that
    \begin{equation}\label{betanu}
    f_\nu(0) =\beta_\nu = \max_{j \in I_{\nu}} \{ a_{j\nu} + g_j(0) \}.
    \end{equation}
Define
    $$
    u_\nu (x) := \alpha_\nu x + \beta_\nu
    $$
for all $x\in\R$. Then by the  by the convexity of the graph of $f_\nu$, it follows that
    $$
    f_\nu(x) \geq u_\nu(x)
    $$
for all $x\in\R$. Therefore,
    \begin{equation}\label{ineq}
    \frac12 f_\nu(r) + \frac12 f_\nu(-r) \geq  \frac12 u_\nu(r) + \frac12 u_\nu(-r) =\beta_\nu.
    \end{equation}
Now, by \eqref{f01} and \eqref{ineq}, we have
    \begin{equation}\label{eqn}
    \begin{split}
    \frac12 v(r)+\frac12 v(-r) &= \frac12\sum_{\nu=n+1}^q \left(f_\nu(r)+f_\nu(-r)\right)\\
    &= \frac12 \sum_{\nu=n+1 \atop \card(I_{\nu})=n+1}^q \left(f_\nu(r)+f_\nu(-r)\right) + \frac12 \sum_{\nu=n+1 \atop \card(I_{\nu})< n+1}^q \left(f_\nu(r)+f_\nu(-r)\right)\\
    &\geq  \frac12 \sum_{\nu=n+1 \atop \card(I_{\nu})=n+1}^q \left(f_\nu(r)+f_\nu(-r)\right) + \sum_{\nu=n+1 \atop \card(I_{\nu})< n+1}^q \beta_\nu\\
    &\geq (q-n-\lambda)\left(T_{g}(r)+ \max_{j\in\{0,\ldots,n\}}\{g_j(0)\}\right)+ \sum_{\nu=n+1 \atop \card(I_{\nu})=n+1}^q \min_{j\in I_{\nu}} \{a_{j\nu}\}\\  &\quad +\sum_{\nu=n+1 \atop \card(I_{\nu})< n+1}^q \max_{j \in I_{\nu}} \{ a_{j\nu} + g_j(0) \},
    \end{split}
    \end{equation}
and so \eqref{v1} holds as we claimed.

By \eqref{vLK} and the tropical Jensen formula \eqref{tPJ}, we have
	\begin{equation}\label{v2}
	\begin{split}
	\frac12 v(r)+\frac12 v(-r) &= \frac12 \tilde L(r)+\frac12 \tilde L(-r) + \frac12 K(r)+\frac12 K(-r)\\
	&= \frac12 \tilde L^+(r)-\frac12(-\tilde L)^+(r)+\frac12 \tilde L^+(-r)-\frac12(-\tilde L)^+(-r)\\
	&\quad+\frac12 K^+(r)-\frac12(-K)^+(r)+\frac12 K^+(-r)-\frac12(-K)^+(-r)\\
	&= m(r,\tilde L)-m(r,1_\circ\oslash\tilde L)+m(r,K)-m(r,1_\circ\oslash K)\\
	&= N(r,1_\circ\oslash\tilde L)-N(r,\tilde L)+m(r,K)-m(r,1_\circ\oslash K)+\tilde L(0).	
	\end{split}
	\end{equation}
Since
    $$
    \tilde L = L \odot \frac{\overline{f}_1 \odot \cdots \odot \overline{f}_n^{[n]}}{f_{1}\odot\cdots\odot f_n}\oslash
    $$
it follows by \eqref{v2} that
    \begin{equation}\label{vvv2}
    \begin{split}
    \frac12 v(r)+\frac12 v(-r) &\leq N(r,1_\circ\oslash\tilde L)-N(r,\tilde L)+m(r,K)-m(r,1_\circ\oslash K) \\
    &\quad + L(0) + \sum_{j=1}^n \left( \max_{0 \leq i \leq n} \{a_{ij}+g_i(j)\}- \max_{0 \leq i \leq n} \{a_{ij}+g_i(0)\} \right).
    \end{split}
    \end{equation}
The function $K$ in \eqref{K} can be written in the form
	\begin{equation*}\label{K2}
	K=\ct(1_\circ,f_1 \oslash f_0,\ldots,f_n \oslash f_0)  \oslash \left( (\overline{f}_1 \oslash \overline{f}_0) \odot \cdots \odot (\overline{f}_n^{[n]}\oslash\overline{f}_0^{[n]})\right),
	\end{equation*}
which implies that $K$ consists purely of tropical sums and
products of the form
$(\overline{f}_{j}^{[l]}\oslash\overline{f}_0^{[l]}) \oslash (\overline{f}_j^{[m]} \oslash \overline{f}_0^{[m]})$ where
$l,m\in\{0,1,\ldots,n\}$ and $j\in\{1,\ldots,q\}$. Therefore,
    \begin{equation}\label{logdappl}
    m(r,K) \leq   \sum_{j=1}^q\sum_{l=0}^n \sum_{m=0}^n m\left(r,(\overline{f}_{j}^{[l]}\oslash\overline{f}_0^{[l]}) \oslash (\overline{f}_j^{[m]} \oslash \overline{f}_0^{[m]})\right).
    \end{equation}
By combining \eqref{eqn}, \eqref{vvv2} and \eqref{logdappl}, it follows that
    \begin{equation}\label{tcartanineqtilde}
    \begin{split}
    (q-n-\lambda)T_g(r) &\leq
    N(r,1_\circ\oslash \tilde L)-N(r,\tilde L) \\ &\quad + \sum_{j=1}^q\sum_{l=0}^n \sum_{m=0}^n m\left(r,(\overline{f}_{j}^{[l]}\oslash\overline{f}_0^{[l]}) \oslash (\overline{f}_j^{[m]} \oslash \overline{f}_0^{[m]})\right)\\
    &\quad + L(0) + \sum_{j=1}^n \left( \max_{0 \leq i \leq n} \{a_{ij}+g_i(j)\}- \max_{0 \leq i \leq n} \{a_{ij}+g_i(0)\} \right) \\
    &\quad - \sum_{\nu=n+1 \atop \card(I_{\nu})=n+1}^q \min_{j\in I_{\nu}} \{ a_{j\nu}\}  -\sum_{\nu=n+1 \atop \card(I_{\nu})< n+1}^q \max_{j \in I_{\nu}} \{ a_{j\nu} + g_j(0) \} \\ & \quad -(q-n-\lambda)\max_{j\in\{0,\ldots,n\}}\{g_j(0)\}.
    \end{split}
    \end{equation}
\end{proof}

By imposing a growth condition, which demands that the hyper-order of the tropical holomorphic curve under consideration is strictly less than one, we can show that the error term in Theorem~\ref{tcartan} is small with respect to $T_g(r)$.

\begin{theorem}\label{tcartan2}
Let $q$ and $n$ be positive integers with $q>n$, and let $\varepsilon>0$.
Given $n+1$ tropical entire functions $g_0, \ldots , g_n$ without common roots, and linearly independent in Gondran-Minoux sense, let the $q+1$ tropical linear combinations $f_{0}, \ldots , f_q$ of the $g_j$ over the semi-ring $\mathbb{R}_{\max}$ be defined by
$$
f_k(x)=a_{0k}\odot g_0(x)\oplus a_{1k}\odot g_1(x)\odot \cdots \odot a_{nk}\odot g_n(x), \quad 0 \leq k \leq q.
$$
Let $
    \lambda=\mbox{ddg}(\{f_{n+1}, \ldots , f_q\})
    $
and
\begin{equation}\label{L2}
L=\frac{f_0\odot f_1 \odot \cdots \odot f_{n}\odot f_{n+1} \odot \cdots \odot f_{q}}{C_{\circ}(f_0, f_1, \cdots , f_n)}\oslash.
\end{equation}
If the tropical holomorphic curve $g$ of $\mathbb{R}$ into $\mathbb{TP}^{n}$ with reduced representation $\mathbf{g}=(g_0, \ldots , g_n)$
is of hyper-order
\begin{equation}\label{growth}
\varsigma:=\varsigma(\mathbf{g})<1,
\end{equation}
then
\begin{equation}\label{tcartanineq2}
\begin{split}
(q-n-\lambda)T_{g}(r) &\leq N\bigl(r, 1_{\circ}\oslash L\bigr)-N(r, L)+o\left(\frac{T_{g}(r)}{r^{1-\varsigma-\varepsilon}}\right),
\end{split}
\end{equation}
where $r$ approaches infinity outside an exceptional set of finite logarithmic measure.
\end{theorem}

Before proving this theorem we briefly return to the value distribution theory by H. Cartan \cite{cartan:33} and Nochka \cite{nochka:83}. The characteristic function of the system $g=(g_0, g_1, \ldots , g_n)$, where $g_0, g_1, \ldots , g_n$ are entire functions in $\mathbb{C}$ without common zeros, is defined by
$$
T_f(r)=\frac{1}{2\pi}\int_0^{2\pi} U(re^{i\theta})d\theta - U(0), \qquad
U(z)=\max_{0\leq j\leq n} \log |g_j(z)|.
$$
Let $X$ be a set of linear combinations of $g_0, g_1, \ldots , g_n$ with coefficients in $\mathbb{C}$ that do not vanish identically and are located in general position.
The number $\lambda$ defined by
$$
\lambda:=\dim \left\{(c_0, c_1, \ldots, c_n)\in\mathbb{C}^{n+1} \, \Big|\,
\sum_{j=0}^n c_j g_j\equiv 0 \right\}
$$
plays an important role in Cartan's and Nochka's second main theorems.
We note that $0\leq \lambda \leq n-1$ and
$$
\dim \left\{(c_0, c_1, \ldots, c_n)\in\mathbb{C}^{n+1} \, \Big|\,
\sum_{j=0}^n c_j f_j\equiv 0 \right\}=\lambda
$$
for any $n+1$ elements $f_0, f_1, \ldots , f_n$ in $X$.
The system $g$ is called {\it degenerated} if $\lambda >0$.
Then the Cartan-Nochka second main theorem states that for any $q+1$ combinations $f_0, f_1, \ldots , f_n$ in $X$,
\begin{equation}\label{2ndmttruncated}
(q-n-\lambda)T_f(r)\leq \sum_{j=0}^{q} N_{n-\lambda}(r, 1/f_j)+O\bigl(\log rT_f(r)\bigr),
\end{equation}
as $r\to \infty$ except for a set of finite linear measure. The counting function $N_{n-\lambda}(r,1/f_j)$ counts the zeros of $f_j$ of multiplicity $m$ so that each such zero is taken into account exactly $\min\{m, n-\lambda\}$ times. The truncated counting functions on the right hand side of \eqref{2ndmttruncated} arise as combinations of regular counting functions and the ramification term expressed in terms of the Wronskian determinant. Although Theorem \ref{tcartan2} has an analogue of the ramification term, finding out how to introduce a consistent truncation for tropical linear combinations of tropical entire functions appears to be difficult. Theorem \ref{tcartan2} implies that
    \begin{equation}\label{2mtwithramification}
    (q-n-\lambda)T_{g}(r) \leq \sum_{j=0}^q N(r, 1_{\circ}\oslash f_j)  - N(r, 1_{\circ}\oslash C_{\circ}(f_0, \ldots , f_n))+ o\left( \frac{T_{g}(r)}{r^{1-\varsigma-\varepsilon}} \right),
    \end{equation}
which is of an analogous form to the classical second main theorem, and it has a ramification-type term $N(r, 1_{\circ}\oslash C_{\circ}(f_0, \ldots , f_n))$. However, this counting function is in terms of the tropical linear combinations $f_0,\ldots,f_n$ of $g_0,\ldots,g_n$, rather than the basis functions themselves, as is the case in the classical Cartan second main theorem. This is essentially due to the properties of tropical matrixes, which in general are not invertible even when they are regular. The interpretation of the term $N(r, 1_{\circ}\oslash C_{\circ}(f_0, \ldots , f_n))$ is far from straightforward and,
as the following example demonstrates, the situation is not clear even in the one-dimensional case.

\begin{example} Let $f_0$ and $f_1$ be tropical entire functions such that
\begin{equation}\label{f0f1}
f_0(x) = \left\{
\begin{array}{ll}
1 & (x\leq 0) \\
x+1 & (x\geq 0)
\end{array}
\right. , \quad \text{and} \quad
f_1(x) = \left\{
\begin{array}{ll}
-(x-1) & (x\leq 1) \\
2(x-1) & (x\geq 1)
\end{array}
\right.
\end{equation}
locally in the interval $x\in [-1,2]$. In particular, one can take $f_0$ and $f_1$ as tropical polynomials defined by \eqref{f0f1} in the whole real line. A simple calculation shows that
\begin{eqnarray*}
C_{\circ}(f_0,f_1)&=&\max\{f_0(x+1)+f_1(x), f_0(x)+f_1(x+1)\}\\
&=&
\left\{ \
\begin{array}{ll}
-x+2 & (x\leq -1)\\
3 & (-1\leq x \leq 0)\\
3  & (0\leq x \leq 2/3)\\
3x+1 & (2/3 \leq x \leq 1)\\
3x+1 & (1\leq  x).
\end{array}
\right.
\end{eqnarray*}
Hence $C_{\circ}(f_0,f_1)$ has roots at $x=-1$ of multiplicity~$1$ and at $x=2/3$ of multiplicity~$3$, while at the points $x=0,1$, where $f_0$ and $f_1$ have roots, it is regular.
Therefore some reduction occurs in multiplicity from points near the roots of $f_0$ and $f_1$, but it is uncertain whether this reduction follows any principle such as the truncation rule by a fixed number as in the Cartan-Nochka theorem. In this example one observes that four points $x=-1,0,1$ and $x=2/3$ contribute to
$$
N(r,1_{\circ}\oslash f_0)+N(r,1_{\circ}\oslash f_1)-N(r, 1_{\circ}\oslash C_{\circ}(f_0,f_1)
$$
by $-1(=0-1)$, $+1=(1-0)$, $3=(3-0)$ and $-3=(0-3)$  respectively.
\end{example}

Similar to the previous example, several simple examples for $n=1$
tell us that the sum of multiplicities of the roots of $f_0$ and $f_1$ equals to the sum of multiplicities of the roots of their tropical Casoratian $C_{\circ}(f_0,f_1)$ exactly.
But at this moment we do not have any proof for a general statement for this phenomenon even when $n=1$ yet.



\begin{proof}[Proof of Theorem \ref{tcartan2}.] Our starting point is the estimate obtained in Theorem~\ref{tcartan}. First of all, we will simplify the error term in \eqref{tcartanineq}. By  Corollary~\ref{logdercor}, it follows that
    \begin{equation*}
    \sum_{j=1}^q\sum_{l=0}^n \sum_{m=0}^n m\left(r,(\overline{f}_{j}^{[l]}\oslash\overline{f}_0^{[l]}) \oslash (\overline{f}_j^{[m]} \oslash \overline{f}_0^{[m]})\right) = o\left(\frac{T(r,f_j \oslash f_0)}{r^{1-\varsigma-\varepsilon}}\right)
    \end{equation*}
for all $r$ outside of an exceptional set of finite logarithmic measure. Hence, by Lemma~\ref{Tlemma} we have
    \begin{equation}\label{Kestim}
    \sum_{j=1}^q\sum_{l=0}^n \sum_{m=0}^n m\left(r,(\overline{f}_{j}^{[l]}\oslash\overline{f}_0^{[l]}) \oslash (\overline{f}_j^{[m]} \oslash \overline{f}_0^{[m]})\right) =o\left(\frac{T_g(r)}{r^{1-\varsigma-\varepsilon}}\right)
    \end{equation}
again, for all $r$ outside of a set of finite logarithmic measure. In addition, by incorporating the constant terms in the right hand side of \eqref{tcartanineq} into the error term $o\left(\frac{T_g(r)}{r^{1-\varsigma-\varepsilon}}\right)$, we have
\begin{equation}\label{tcartanineq2tilde}
\begin{split}
(q-n-\lambda)T_{g}(r) &\leq N\bigl(r, 1_{\circ}\oslash \tilde L\bigr)-N(r, \tilde L)+o\left(\frac{T_{g}(r)}{r^{1-\varsigma-\varepsilon}}\right)
\end{split}
\end{equation}
outside of an exceptional set of finite logarithmic measure. We will now go on to show that \eqref{tcartanineq2tilde} implies \eqref{tcartanineq2}. Since $g_0,\ldots,g_n$ have no roots that are common to all of them, there will be no cancelation for at least one $g_k\oslash f_j$, $k=0,\ldots,n$, at each root of $f_j$. Thus we have
    \begin{equation}\label{Njgrowth0}
    \begin{split}
    N\left(r,1_\circ \oslash f_j\right) &\leq \sum_{k=0}^n N\left(r,g_k \oslash f_j\right)\\
     & \leq\sum_{k=0}^n T\left(r,g_k \oslash f_j\right)\\
     &\leq (n+1)T_{g}(r)+O(1)
    \end{split}
    \end{equation}
by applying Lemma \ref{Tlemma}. Hence
    \begin{equation}\label{Njgrowth}
    N\left(r,1_\circ \oslash f_j\right) \leq (n+2)T_{g}(r)
    \end{equation}
for all $r$ large enough, and so the assumption $\varsigma(g)<1$ yields
    $$
    \eta_j := \limsup_{r\to\infty} \frac{\log\log N\left(r,1_\circ \oslash f_j\right)}{\log r}<1
    $$
for all $j=1,\ldots,n$. Therefore, by Lemma \ref{technical}, we have
    \begin{equation}\label{Njgrowth2}
    N\left(r,1_\circ \oslash \overline{f}_j^{[j]}\right) \leq N\left(r+j,1_\circ \oslash f_j\right)  = N\left(r,1_\circ \oslash f_j\right) + o\left(\frac{N\left(r,1_\circ \oslash f_j\right)}{r^{1-\eta_j-\varepsilon}}\right),
    \end{equation}
where $j=1,\ldots,n$ and $r$ tends to infinity outside of an exceptional set of finite logarithmic measure. By combining \eqref{Njgrowth} and \eqref{Njgrowth2}, it follows that
    \begin{equation}\label{Nj}
    N\left(r,1_\circ \oslash \overline{f}_j^{[j]}\right) \leq N\left(r,1_\circ \oslash f_j\right)  + o\left(\frac{T_{g}(r)}{r^{1-\varsigma-\varepsilon}}\right), \qquad j=1,\ldots,n,
    \end{equation}
outside of an exceptional set of finite logarithmic measure. Therefore,
    \begin{equation*}
    \begin{split}
    &N\left(r,1_\circ \oslash \tilde L\right)-N(r,\tilde L) \\ &\quad= N\left(r,\frac{1_\circ}{f_0\odot\overline{f}_1\odot\cdots\odot \overline{f}^{[n]}_n\odot f_{n+1}\odot\cdots
    \odot f_q}\oslash\right) - N\left(r,\frac{1_\circ}{C_\circ(f_0,\ldots,f_n)}\oslash\right) \\
    &\quad = \sum_{j=0}^n N\left(r,1_\circ\oslash \overline{f}^{[j]}_j\right) + N\left(r,\frac{1_\circ}{f_{n+1}\odot\cdots
    \odot f_q}\oslash\right) - N\left(r,\frac{1_\circ}{C_\circ(f_0,\ldots,f_n)}\oslash\right)\\
    &\quad \leq  \sum_{j=0}^n N\left(r,1_\circ \oslash f_j\right)  + N\left(r,\frac{1_\circ}{f_{n+1}\odot\cdots\odot
    f_q}\oslash\right) \\ &\quad\quad - N\left(r,\frac{1_\circ}{C_\circ(f_0,\ldots,f_n)}\oslash\right) + o\left(\frac{T_{g}(r)}{r^{1-\varsigma-\varepsilon}}\right)
    \\
    &\quad = N\left(r,\frac{1_\circ}{f_0\odot\cdots\odot f_q}\oslash\right) - N\left(r,\frac{1_\circ}{C_\circ(f_0,\ldots,f_n)}\oslash\right) + o\left(\frac{T_{g}(r)}{r^{1-\varsigma-\varepsilon}}\right)\\
    &\quad = N\left(r,1_\circ\oslash L\right)-N(r,L) + o\left(\frac{T_{g}(r)}{r^{1-\varsigma-\varepsilon}}\right)
    \end{split}
    \end{equation*}
for all $r$ outside of an exceptional set of finite logarithmic measure. The assertion follows by combining the above inequality with \eqref{tcartanineq2tilde}. 
\end{proof}


\section{Tropical second main theorem in the one dimensional case}

Laine and Tohge proved the following version of the second main theorem for tropical meromorphic functions.

\begin{theorem}[\cite{lainet:11}]\label{2mtLT}
If $f$ is a non-constant tropical meromorphic function of hyper-order $\varsigma<1$, if $\varepsilon>0$, and if $q\geq 1$ distinct values $a_1,\ldots,a_q\in\R$ satisfy
    \begin{equation}\label{2mtLTassumption}
    \max\{a_1,\ldots,a_q\}< \inf\{f(\alpha):\omega_f(\alpha)<0\},
    \end{equation}
and
    \begin{equation}\label{2mtLTassumption2}
    \inf\{f(\alpha):\omega_f(\alpha)>0\}>-\infty,
    \end{equation}
then
    \begin{equation}\label{2mtLTassertion}
    qT(r,f)\leq \sum_{k=1}^q N(r,1_\circ\oslash(f\oplus a_k))+o\left(\frac{T(r,f)}{r^{1-\varsigma-\varepsilon}}\right)
    \end{equation}
for all $r$ outside of an exceptional set of finite logarithmic measure.
\end{theorem}

We will now show that our main result implies Theorem \ref{2mtLT}, and we can even drop the assumption \eqref{2mtLTassumption2}. The fact that one can replace the inequality in \eqref{2mtLTassertion} with an equality follows from the fact that $N(r,1_\circ\oslash(f\oplus a_k)) \leq T(r,f)$ for all $k\in\{1,\ldots,q\}$.

\begin{corollary}\label{cor}
If $f$ is a non-constant tropical meromorphic function of hyper-order $\varsigma<1$, if $\varepsilon>0$, and if $q\geq 1$ distinct values $a_1,\ldots,a_q\in\R$ satisfy
    \begin{equation}\label{2mtLTassumption_cor}
    \max\{a_1,\ldots,a_q\}< \inf\{f(\alpha):\omega_f(\alpha)<0\},
    \end{equation}
then
    \begin{equation}\label{2mtLTassertion_cor}
    qT(r,f) = \sum_{k=1}^q N(r,1_\circ\oslash(f\oplus a_k))+o\left(\frac{T(r,f)}{r^{1-\varsigma-\varepsilon}}\right)
    \end{equation}
for all $r$ outside of an exceptional set of finite logarithmic measure.
\end{corollary}

By choosing $q=1$ in Corollary~\ref{cor}, it follows that
    \begin{equation*}
    T(r,f) =  N(r,1_\circ\oslash(f\oplus a))+o\left(\frac{T(r,f)}{r^{1-\varsigma-\varepsilon}}\right)
    \end{equation*}
provided that $a\in\R$ satisfies
    \begin{equation}\label{acond}
    a < \inf\{f(\alpha):\omega_f(\alpha)<0\}.
    \end{equation}
This means that there is in practise no room left for any meaningful ramification in Corollary~\ref{cor}, and any possible ramification would have to be included in the error term $o\left(\frac{T(r,f)}{r^{1-\varsigma-\varepsilon}}\right)$.

We have managed to drop the condition \eqref{2mtLTassumption2} from Corollary~\ref{cor} by using Theorem~\ref{tcartan2}, but the assumption \eqref{2mtLTassumption_cor} cannot be deleted. Namely, by omitting \eqref{2mtLTassumption_cor} we can arrive at cases such that $f\oplus a_k\equiv a_k$ for which $$
N(r,1_\circ\oslash(f\oplus a_k)) = O(1)
$$
and hence the assertion of Corollary~\ref{cor} fails to be valid.


%

\begin{proof}[Proof of Corollary \ref{cor}.]  We will consider two different cases. First, we assume that
    \begin{equation}\label{2mtLTassumption_cor2}
    f\not\equiv f\oplus a_k
    \end{equation}
for all $k=1,\ldots,q$. By Proposition \ref{fgh}, the function $f$ can be represented in the form $f=g_1\oslash g_0$, where $g_0$ and $g_1$ are tropical entire functions without common roots. Since by assumption $f$ is non-constant, it follows that $g_0$ and $g_1$ are linearly independent in the Gondran-Minoux sense. Let $n=1$ in Theorem \ref{tcartan2}, and put $g=[g_0:g_1]$. Further, let $f_0= g_0$, $f_1= g_1$, and
\begin{equation}\label{fk}
f_k = (a_{k-1}\odot g_0)\oplus (1_\circ \odot g_1)
\end{equation}
for $k\in\{2,\ldots,q+1\}$. Then
\begin{eqnarray*}
L(x)  &=& f_2(x)\odot \cdots \odot f_{q+1}(x) \oslash \big( (f_1(x+1)\oslash f_1(x))\oplus (f_0(x+1) \oslash f_0(x))\big)
\end{eqnarray*}
since
\begin{eqnarray*}
C_{\circ}(f_0,f_1) &=& (f_0(x)\odot f_1(x+1))\oplus (f_0(x+1)\odot f_1(x))\\
                   &=& (f_1(x+1)\oslash f_1(x))\oplus(f_0(x+1) \oslash f_0(x))\odot f_0(x) \odot f_1(x).
\end{eqnarray*}
Denoting $D(x):=(f_1(x+1)\oslash f_1(x))\oplus(f_0(x+1) \oslash f_0(x))$, we have
$$
N(r, 1_\circ\oslash L)\leq \sum_{k=2}^{q+1} N\bigl(r,1_\circ\oslash f_k\bigr)+N(r,D)
$$
and
$$
N(r, L)= N(r,1_\circ\oslash D)
$$
so that
\begin{equation}\label{mrD}
\begin{split}
N(r, 1_\circ\oslash L)-N(r,L)
&\leq  \sum_{k=2}^{q+1} N\bigl(r,1_\circ\oslash f_k \bigr)
 +N(r,D)-N(r,1_\circ\oslash D)\\
&= \sum_{k=2}^{q+1} N\bigl(r,1_\circ\oslash f_k\bigr)
 - m(r,D)+m(r,1_\circ\oslash D)+D(0),
\end{split}
\end{equation}
where $D(0)=\max\{g_1(1)-g_1(0),g_0(1)-g_0(0)\}$ depends only on the function $f$ and not on the values $a_1,\ldots,a_q$. Here we applied the tropical Jensen formula to the function $D(x)$. Moreover, since
    $$
    1_\circ\oslash D(\pm r) = \min\{f_1(\pm r)-f_1(\pm r+1),f_0(\pm r)-f_0(\pm r+1)\} \leq f_0(\pm r)\oslash f_0(\pm r+1),
    $$
say (the upper bound could have been as easily taken in terms of $f_1$), it follows that
    \begin{equation}\label{mrDerror}
    \begin{split}
    m(r,&1_\circ\oslash D) - m(r,D) \\ &= \frac{1}{2}(1_\circ\oslash D)(r)+\frac{1}{2}(1_\circ\oslash D)(-r)\\
    &\leq \frac{1}{2}(f_0(r)\oslash f_0(r+1))+\frac{1}{2}(f_0(-r)\oslash f_0(-r+1)) \\
    &= m(r,f_0(x)\oslash f_0(x+1)) - m(r,f_0(x+1)\oslash f_0(x))\\
    &= N(r,f_0(x+1)\oslash f_0(x)) - N(r,f_0(x)\oslash f_0(x+1)) + f_0(0)\oslash f_0(1)\\
    &=N(r,1_\circ\oslash f_0(x)) - N(r,1_\circ\oslash f_0(x+1)) + f_0(0)\oslash f_0(1) = o\left(\frac{T_{g}(r)}{r^{1-\varsigma-\varepsilon}}\right),
    \end{split}
    \end{equation}
where the last asymptotic equation follows by a similar reasoning as in \eqref{Nj}. By combining \eqref{mrD} and \eqref{mrDerror}, it follows that
    \begin{equation}\label{NrD}
    N(r, 1_\circ\oslash L)-N(r,L) \leq \sum_{k=2}^{q+1} N\bigl(r,1_\circ\oslash f_k \bigr) + o\left(\frac{T_{g}(r)}{r^{1-\varsigma-\varepsilon}}\right),
    \end{equation}
where $r$ tends to infinity outside of an exceptional set of finite logarithmic measure.

By assumption \eqref{2mtLTassumption_cor}, it follows that the roots of $g_0$ are exactly the poles of $a_k\oplus f$, counting multiplicity, for all $k=1,\ldots,q$. Therefore, by \eqref{fk} we can see that
    \begin{equation}\label{NN}
    N\bigl(r,1_\circ\oslash f_k\bigr) = N(r,1_\circ\oslash(f\oplus a_{k-1}))
    \end{equation}
for all $k=2,\ldots,q+1$. Moreover, assumptions \eqref{2mtLTassumption_cor} and \eqref{2mtLTassumption_cor2} imply that
    \begin{equation*}
    f\not\equiv f\oplus a_k\not\equiv a_k \qquad \textrm{for all } k\in\{1,\ldots,q\},
    \end{equation*}
which is the same as
    \begin{equation}\label{2mtLTassumptioncombined}
    g_1\oslash g_0\not\equiv f_{k+1}\oslash g_0 \not\equiv a_k \qquad \textrm{for all } k\in\{1,\ldots,q\}.
    \end{equation}
Since $f$ is non-constant by assumption, it follows by the condition \eqref{2mtLTassumption_cor2} that the set $\{f_2,\ldots,f_{q+1}\}$ of tropical linear combinations of $g_0$ and $g_1$ is non-degenerate, and so
by Theorem \ref{tcartan2} with $n=1$ and $\lambda=0$, and using \eqref{NrD}, we have
	\begin{equation}\label{Tgform}
	q T_{g}(r)\leq \sum_{k=2}^{q+1} N\bigl(r,1_\circ\oslash f_k\bigr)+o\left(\frac{T_{g}(r)}{r^{1-\varsigma-\varepsilon}}\right),
	\end{equation}
as $r$ tends to infinity outside of an exceptional set of finite logarithmic measure. Finally, by applying Proposition~\ref{consistency} with $f=g_0\oslash g_1$ and $g=[g_0:g_1]$, and using \eqref{NN}, equation \eqref{Tgform} becomes
	$$
	q T(r,f)\leq \sum_{k=1}^q N\bigl(r,1_\circ\oslash(f\oplus a_k)\bigr)+o\left(\frac{T(r,f)}{r^{1-\varsigma-\varepsilon}}\right),
	$$	
where the exceptional set is the same as in \eqref{Tgform}.

Suppose finally that \eqref{2mtLTassumption2} is not satisfied for some $a_{i}\in\{a_1,\ldots,a_q\}$. Then, for all $a_i$ for which \eqref{2mtLTassumption2} is not satisfied, we have
    \begin{equation*}
    f\equiv f\oplus a_i,
    \end{equation*}
which means that
    \begin{equation*}
    a_i \leq f(x)
    \end{equation*}
for all $x\in\R$. Therefore the function $-f=1_\circ\oslash f$ is bounded above by $-a_i$, and so by the definition of the tropical proximity function,
    \begin{equation*}
    m(r,1_\circ\oslash f) \leq \max\{-a_i,0\}.
    \end{equation*}
Hence, by the tropical Jensen formula,
    \begin{equation*}
    \begin{split}
    T(r,f) &= T(r,1_\circ\oslash f) + f(0)\\ &\leq N(r,1_\circ\oslash f) +  \max\{-a_i,0\} +f(0) \\
    &= N(r,1_\circ\oslash (f\oplus a_i)) +  O(1),
    \end{split}
    \end{equation*}
which is the required estimate. By adding up the corresponding estimates for all $a_i$ such that \eqref{2mtLTassumption2} is not satisfied, and applying the first part of the proof for the rest, we obtain the assertion of the corollary. 
\end{proof}

\begin{remark} Let us recall the second main theorem in the classical Nevanlinna theory for meromorphic functions in the plane~$\mathbb{C}$.
This states that the estimate
$$
(q-1)T(r,f)\leq \overline{N}(r,f)+\sum_{k=1}^q \overline{N}\bigl(r,1/(f-a_k)\bigr)+O\bigl(\log rT(r,f)\bigr)
$$
holds for all $r$ outside of an exceptional set of finite measure.
Here $a_1, \ldots , a_q \in\mathbb{C}$ are $q\geq 1$ distinct values.
The counting function $\overline{N}(r,\phi)$ counts distinct poles of the meromorphic function $\phi$ under consideration, where the multiplicity is ignored.
By definition, the function possesses the feature that distinguishes this and other counting functions, that is,
$$
\overline{N}(r,\phi)\equiv \overline{N}(r, \phi^{(j)}) \quad \text{for~all~} j\in\mathbb{N}.
$$
On the other hand, we have looked in the proof of Corollary~\ref{cor} as an application of Lemma~\ref{technical}, thanks to the assumption $\varsigma(\psi)<1$ on a tropical meromorphic function $\psi$, the estimate
$$
N\bigl(r,\psi(x)\bigr)= N\bigl(r, \psi(x+j)\bigr)+o\left(\frac{T(r,\psi)}{r^{1-\varsigma-\varepsilon}}\right) \quad \text{for~all~} j\in\mathbb{N},
$$
holds for any $\varepsilon>0$ as $r\to\infty$ outside of an exceptional set of finite logarithmic measure.
We might think that this is a sign of a possibility to incorporate a more exact estimate including a concept of truncation for the `multiplicity' by means of the shift operator.
\end{remark}

\def\cprime{$'$}


\end{document}